\documentclass{article}
\usepackage[utf8]{inputenc}
\usepackage[margin=1in]{geometry}
\usepackage{graphicx}
\usepackage{amssymb}
\usepackage{amsthm}
\usepackage{epstopdf}
\usepackage{dsfont}
\usepackage{xcolor}
\usepackage{amsmath}
\usepackage{mathrsfs}
\usepackage[parfill]{parskip}
\usepackage{comment}
\usepackage{tikz}
\usepackage{tikz-3dplot}
\usetikzlibrary{intersections}
\usepackage{mathtools}
\usepackage{hyperref}
\usepackage{caption}
\newtheorem{theorem}{Theorem}[section]
\newtheorem{proposition}[theorem]{Proposition}
\newtheorem{lemma}[theorem]{Lemma}
\newtheorem{corollary}[theorem]{Corollary}
\theoremstyle{definition}
\newtheorem{definition}[theorem]{Definition}

\newtheorem{conjecture}[theorem]{Conjecture}

\newtheorem{remark}[theorem]{Remark}

\begin{document}

\title{Lower Bounds on Face Numbers of Polytopes with $m$ Facets}

\author{Joshua Hinman\\
\small Department of Mathematics\\
\small University of Washington\\
\small Seattle, WA 98195-4350, USA\\
\small \texttt{joshrh@uw.edu}}
\date{}

\maketitle

\begin{abstract}
    Let $P$ be a convex $d$-polytope and $0 \leq k \leq d-1$. In 2023, this author proved the following inequalities, resolving a question of B\'ar\'any:
    \[
    \frac{f_k(P)}{f_0(P)} \geq \frac{1}{2}\biggl[{\lceil \frac{d}{2} \rceil \choose k} + {\lfloor \frac{d}{2} \rfloor \choose k}\biggr],
    \qquad
    \frac{f_k(P)}{f_{d-1}(P)} \geq \frac{1}{2}\biggl[{\lceil \frac{d}{2} \rceil \choose d-k-1} + {\lfloor \frac{d}{2} \rfloor \choose d-k-1}\biggr].
    \]
    We show that for any fixed $d$ and $k$, these are the tightest possible linear bounds on $f_k(P)$ in terms of $f_0(P)$ or $f_{d-1}(P)$. We then give a stronger bound on $f_k(P)$ in terms of the Grassmann angle sum $\gamma_k^2(P)$. Finally, we prove an identity relating the face numbers of a polytope with the behavior of its facets under a fixed orthogonal projection of codimension two.
\end{abstract}

\section{Introduction}

Given the number of facets of a $d$-polytope, or the number of vertices, what is the range of possible values for its other face numbers? This question is at least half-answered: the Upper Bound Theorem, proven by McMullen in 1970, gives the maximum face numbers of a $d$-polytope with $n$ vertices \cite{mcmullen70}. These maxima are achieved by \emph{neighborly} $d$-polytopes, including the \emph{cyclic polytope} $C(d,n)$.

We concern ourselves with the other half of the question: if $P$ is a $d$-polytope with $m$ facets, how small can its other face numbers be? For $P$ simple, we get the following answer from Barnette's Lower Bound Theorem of 1973 \cite{barnette73}. The polar dual $P^*$ of $P$ is a simplicial $d$-polytope with $m$ vertices. Thus,
\[
    f_0(P) = f_{d-1}(P^*) \geq (d-1)m - (d+1)(d-2),
\]
and for $1 \leq k \leq d-2$,
\[
    f_k(P) = f_{d-k-1}(P^*) \geq {d \choose k+1}m - {d+1 \choose k+1}(d-k-1).
\]

Likewise, for $P$ simplicial, Kalai determined the minimum possible values of $f_k(P)$ in 1991 \cite{kalai91}. This work utilized the \emph{$g$-theorem}, a set of necessary and sufficient conditions for the $f$-vector of a simplicial $d$-polytope, proven in 1980 by Billera and Lee \cite{billera80,billera81} (sufficiency) and Stanley \cite{stanley80} (necessity). Among Kalai's 1991 results was the Generalized Upper Bound Theorem: if $P$ is a simplicial $d$-polytope and $f_{d-1}(P) \geq f_{d-1}(C(d,n))$ for some $n$, then $f_k(P) \geq f_k(C(d,n))$ for all $0 \leq k \leq d-1$. In other words, if there exists a neighborly $d$-polytope with $m$ facets, then it simultaneously minimizes all other face numbers among simplicial $d$-polytopes with $m$ facets. However, for general $m$, there is not necessarily a single polytope which minimizes all other face numbers simultaneously.

For general $d$-polytopes with $m$ facets, our knowledge is concentrated in small values of $m$. In 2021, Xue determined the minimum possible face numbers for $m \leq 2d$ \cite{xue20}, proving a 1967 conjecture of Gr\"unbaum \cite{grunbaum68}. Xue then extended her results to the case of $m = 2d+1$ \cite{xue23}.

Very recently, this author proved the following face number relations \cite{hinman23}, resolving a question of B\'ar\'any \cite{barany98}. For all integers $0 \leq k < d$, define
\[
\rho(d, k) = \frac{1}{2}\biggl[{\lceil \frac{d}{2} \rceil \choose k} + {\lfloor \frac{d}{2} \rfloor \choose k}\biggr].
\]

\begin{theorem}[Hinman]
\label{bound}
Let $P$ be a $d$-polytope, and suppose $0 \leq k \leq d-1$. Then
\begin{align}
\frac{f_k(P)}{f_0(P)} &\geq \rho(d, k), \label{eqn:bound}\\
\frac{f_k(P)}{f_{d-1}(P)} &\geq \rho(d, d-k-1). \label{eqn:dual}
\end{align}
In the former, equality holds precisely when $k=0$ or when $k=1$ and $P$ is simple. In the latter, equality holds precisely when $k=d-1$ or when $k=d-2$ and $P$ is simplicial.
\end{theorem}

Relation (\ref{eqn:dual}) gives a lower bound on $f_k(P)$ in terms of $f_{d-1}(P) = m$, but it is only sharp in the case of $k \geq d-2$. For general $d,k$ and $m>2d+1$, the minimum possible value of $f_k(P)$ remains unknown. The strongest conjecture in this area is Kalai's \emph{Generalized Upper Bound Conjecture}, which would extend the Generalized Upper Bound Theorem from simplicial to general polytopes \cite{kalai91}.

\begin{conjecture}[Generalized Upper Bound Conjecture]
\label{gubc}
    Let $P$ be a $d$-polytope with $f_{d-1}(P) \geq f_{d-1}(C(d, n))$. Then for all $0 \leq k \leq d-1$, $f_k(P) \geq f_k(C(d, n))$.
\end{conjecture}

This note is a continuation of the author's work in \cite{hinman23}. Section \ref{s:preliminaries} provides background on polytopes, faces, and solid angles, as well as the more general \emph{Grassmann angles}. In Section \ref{s:tightness}, we prove that for any fixed $d$ and $k$, (\ref{eqn:dual}) is the tightest possible linear bound on $f_k(P)$ in terms of $f_{d-1}(P)$. Specifically, we show that for neighborly polytopes with increasingly many vertices, $f_k(P)/f_{d-1}(P)$ asymptotically approaches $\rho(d,d-k-1)$. In Section \ref{s:nonlinear}, we give a stronger bound on $f_k(P)$ in terms of both $f_{d-1}(P)$ and the Grassmann angle sum $\gamma_k^2(P)$. Finally, we prove a relation between the face numbers of a polytope and those of its facets under a fixed codimension two projection.


\section{Preliminaries}
\label{s:preliminaries}
This section introduces the concepts and prior results which we will use to study extremal problems on face numbers. We will discuss some significant families of polytopes and their faces (\S \ref{ss:faces}), solid angles of polytopes (\S \ref{ss:angles}), and the more general notion of Grassmann angles (\S \ref{ss:grassmann}).

\subsection{Polytopes and Their Faces}
\label{ss:faces}
We begin by discussing some noteworthy classes of polytopes: \emph{simplicial}, \emph{simple}, and \emph{neighborly}. We give special attention to \emph{cyclic polytopes}, a family of neighborly polytopes which we will use to prove that Theorem \ref{bound} gives tight linear bounds. The reader may refer to \cite{grunbaum03,ziegler95} for any undefined terminology.

\begin{definition}
    A \emph{polytope} is the convex hull of finitely many points in a real vector space. For nonnegative integers $d$, a \emph{$d$-polytope} is a polytope whose affine hull has dimension $d$. We consider the empty set a $(-1)$-polytope.
\end{definition}

\begin{definition}
    Let $P$ be a $d$-polytope. A \emph{face} of $P$ is either $P$ itself or a polytope $H \cap P$, where $H$ is a codimension one hyperplane in $\operatorname{aff}(P)$ which does not intersect the interior of $P$. A \emph{vertex} of $P$ is a zero-dimensional face; a \emph{facet} is a $(d-1)$-dimensional face. For $0 \leq k \leq d-1$, we define $f_k(P)$ as the number of $k$-dimensional faces of $P$.
\end{definition}

\begin{definition}
    A polytope is \emph{simplicial} if each of its facets is a simplex. A polytope is \emph{simple} if its polar dual is simplicial.
\end{definition}

\begin{definition}
    Let $P$ be a simplicial $d$-polytope with vertex set $V$. We say $P$ is \emph{neighborly} if for each subset $U \subset V$ with $|U| \leq \lfloor \frac{d}{2} \rfloor$, there exists a face of $P$ whose vertex set is exactly $U$.
\end{definition}

\begin{definition}
    Let $n > d \geq 2$, and let $\gamma:\mathds{R} \to \mathds{R}^d$ be the moment curve, defined as
    \[
        \gamma(t) = (t, t^2, \ldots, t^d).
    \]
    The \emph{cyclic polytope} $C(d,n) \subset \mathds{R}^d$ is the convex hull of $n$ arbitrary, distinct points in the image of $\gamma$.
\end{definition}

The cyclic polytope is simplicial and neighborly, and its facets are determined by the \emph{Gale evenness condition} (see \cite{gale63}). As a result, the combinatorial type of $C(d,n)$ does not depend on our choice of points on the moment curve.

Since $C(d,n)$ is neighborly, we may observe that for all $0 \leq k < \lfloor \frac{d}{2} \rfloor$,
\[
    f_k(P) = {n \choose k+1}.
\]
The remaining face numbers of $C(d,n)$ are determined by the Dehn--Sommerville equations. The following formulas can be found in \cite{billera97}.

\begin{theorem}
\label{eqn:cyclic}
Let $C(d,n)$ be the cyclic $d$-polytope on $n$ vertices. Then for all $0 \leq k \leq d-1$,
\[
    f_k(C(d,n)) = \frac{n - \delta(n-k-2)}{n-k-1} \sum_{j=0}^{\lfloor d/2 \rfloor} {n-1-j \choose k+1-j}{n-k-1 \choose 2j-k-1+\delta},
\]
where $\delta = \lceil d/2 \rceil - \lfloor d/2 \rfloor$. In particular,
\[
    f_{d-1}(C(d,n)) = {n - \lfloor \frac{d+1}{2} \rfloor \choose n-d} + {n - \lfloor \frac{d+2}{2} \rfloor \choose n-d}.
\]
\end{theorem}

Note that for any neighborly, simplicial $d$-polytope $P$ on $n$ vertices and all $0 \leq k \leq d-1$, $f_k(P) = f_k(C(d,n))$.

\subsection{Solid Angles}
\label{ss:angles}
Next, we discuss the solid angles of polytopes, a generalization of plane angles to higher dimensions.

\begin{definition}
    Let $P$ be a $d$-polytope and $G$ a face of $P$. Let $B$ be a ball with center in the relative interior of $G$, intersecting exactly the faces of $P$ which contain $G$. We define the \emph{solid angle} $\varphi(P,G)$ as
    \[
        \varphi(P,G) = \frac{\lambda(B \cap P)}{\lambda(P)},
    \]
    where $\lambda$ is the $d$-dimensional Lebesgue volume. For $0 \leq k \leq d-1$, if $\mathcal{G}_k$ is the set of $k$-dimensional faces of $P$, we define
    \[
        \varphi_k(P) = \sum_{G \in \mathcal{G}_k} \varphi(P,G).
    \]
\end{definition}

We may view solid angles as probabilities, as discussed by Feldman and Klain \cite{feldman09}. Let $P \subset \mathds{R}^d$ be a $d$-polytope and $G$ a face of $P$. If $H \subset \mathds{R}^d$ is a codimension one hyperplane chosen uniformly at random, and $\pi:\mathds{R}^d \to H$ is the orthogonal projection map, then $1-2\varphi(P,G)$ is the probability that $\pi(G)$ is a proper face of $\pi(P)$. Summing over all $k$-dimensional faces yields the following identity; see \cite[(5)]{feldman09}.

\begin{theorem}
\label{feldman}
    Let $P \subset \mathds{R}^d$ be a $d$-polytope and $0 \leq k \leq d-1$. Suppose $H \subset \mathds{R}^d$ is a codimension one hyperplane chosen uniformly at random, and let $\pi:\mathds{R}^d \to H$ be the orthogonal projection map. Then
    \[
        f_k(P) - 2\varphi_k(P) = Ef_k(\pi(P)),
    \]
    where $E$ is the expected value.
\end{theorem}

We conclude our discussion of solid angles with one additional theorem from \cite{hinman23}. Recall that for all integers $0 \leq k \leq d-1$, we define
\[
\rho(d, k) = \frac{1}{2}\biggl[{\lceil \frac{d}{2} \rceil \choose k} + {\lfloor \frac{d}{2} \rfloor \choose k}\biggr].
\]

\begin{theorem}[Hinman]
\label{angles}
For all $(d-1)$-polytopes $Q$ and all $0 \leq k \leq d-2$,
\[
\varphi_k(Q) \geq \rho(d, d-k-1).
\]
\end{theorem}

This theorem was the key ingredient in proving Theorem \ref{bound}, and we will use it similarly here to prove a stronger version of Theorem \ref{bound}.

\subsection{Grassmann Angles}
\label{ss:grassmann}
We next turn our attention to Grassmann angles, introduced by Gr\"unbaum in 1968 \cite{grunbaum68}. These are a generalization of solid angles which further describe a polytope's geometric behavior at one of its faces.

\begin{definition}
    Let $P$ be a $d$-polytope, $G$ a face of $P$, $z$ a point in the relative interior of $G$, and $1 \leq m \leq d-1$. The \emph{Grassmann angle} $\gamma^m(P,G)$ is the probability that for an $m$-dimensional subspace $S \subset \mathds{R}^d$ chosen uniformly at random,
    \[
        (S+z) \cap P = \{z\}.
    \]
    For $0 \leq k \leq d-1$, if $\mathcal{G}_k$ is the set of $k$-dimensional faces of $P$, we define
    \[
        \gamma_k^m(P) = \sum_{G \in \mathcal{G}_k} \gamma^m(P,G).
    \]
\end{definition}

Note that $\gamma^m(P,G)$ is independent of our choice of $z$, and $\gamma^m(P,G)=0$ if $\dim G > d-m-1$.

Like solid angles, Grassmann angles can be understood in terms of projections. Let $P$ be a $d$-polytope and $1 \leq m \leq d-1$. Let $H \subset \mathds{R}^d$ be a codimension $m$ hyperplane chosen uniformly at random, and let $\pi:\mathds{R}^d \to H$ be the orthogonal projection map. Then for each face $G$ of $P$, $\pi(G)$ is a proper face of $\pi(P)$ with probability $\gamma^m(P,G)$. In particular, $\gamma^1(P,G)$ is $1-2\varphi(P,G)$; see Theorem \ref{feldman} and our surrounding discussion.

More generally, linearity of expectation gives us the following analogue of Theorem \ref{feldman} for Grassmann angles.

\begin{lemma}
    \label{grass_expectation}
    Let $P \subset \mathds{R}^d$ be a $d$-polytope, $1 \leq m \leq d-1$, and $0 \leq k \leq d-1$. Let $H \subset \mathds{R}^d$ be a codimension $m$ hyperplane chosen uniformly at random, and let $\pi:\mathds{R}^d \to H$ be the orthogonal projection map. Then
    \[
        \gamma_k^m(P) = Ef_k(\pi(P)),
    \]
    where $E$ is the expected value.
\end{lemma}

Of particular interst is the Grassmann angle $\gamma^2(P,G)$. The following result of Gr\"unbaum \cite{grunbaum68} identifies $\gamma^2(P,G)$ with the \emph{angle deficiency} of $P$ at $G$.

\begin{theorem}[Gr\"unbaum]
\label{deficiency}
    Let $P$ be a $d$-polytope with $d \geq 3$, and let $G$ be a face of $P$. Let $\mathcal{F}_G$ be the set of facets of $P$ containing $G$. Then
    \[
        \gamma^2(P,G) = 1-\sum_{F \in \mathcal{F}_G}\varphi(F,G).
    \]
\end{theorem}


\section{Proof of Tightness}
\label{s:tightness}

In this section, we will prove that the bounds in Theorem \ref{bound} are tight. That is, for any fixed dimension $d$ and $0 \leq k \leq d-1$, Theorem \ref{bound} gives the greatest constant lower bounds on $f_k(P)/f_0(P), f_k(P)/f_{d-1}(P)$ for $d$-polytopes $P$. We will prove this using the cyclic polytope $C(d,n)$; specifically, we will show that as $n$ grows arbitrarily large, $f_k(C(d,n))/f_{d-1}(C(d,n))$ asymptotically approaches $\rho(d, d-k-1)$.

For convenience, we will write $f_k(C(d,n))$ to denote the value given by Theorem \ref{eqn:cyclic} even when $d < 2$ or $k > d-1$.

\begin{lemma}
\label{cyclic}
    For all $n > d \geq 2$ and $0 \leq k \leq d-1$,
    \[
        f_k(C(d,n)) = \rho(d, d-k-1)f_{d-1}(C(d,n)) + f_k(C(d-2,n)).
    \]
\end{lemma}

\begin{proof}
Let $n > d \geq 2$ and $0 \leq k \leq d-1$. If $k < \lfloor \frac{d}{2} \rfloor - 1$, then $\rho(d, d-k-1) = 0$ and $f_k(C(d,n)) = f_k(C(d-2,n)) = {n \choose k+1}$, so we are done.

Suppose $k \geq \lfloor \frac{d}{2} \rfloor - 1$. By Theorem \ref{eqn:cyclic},
\begin{align*}
    f_k(C(d,n)) - f_k(C(d-2,n)) &= \frac{n - \delta(n-k-2)}{n-k-1} {n - \lfloor\frac{d}{2}\rfloor - 1 \choose n-k-2}{n-k-1 \choose n-d}\\
    &= \frac{n - \delta(n-k-2)}{n - \lfloor\frac{d}{2}\rfloor} {n - \lfloor \frac{d}{2} \rfloor \choose n-k-1}{n-k-1 \choose n-d} \\
    &= \frac{n - \delta(n-k-2)}{n - \lfloor\frac{d}{2}\rfloor} {n - \lfloor \frac{d}{2} \rfloor \choose n-d}{\lceil \frac{d}{2} \rceil \choose d-k-1} \\
    &= \frac{n - \delta(n-k-2)}{\lceil \frac{d}{2} \rceil} {n - \lfloor\frac{d}{2}\rfloor - 1 \choose n-d}{\lceil \frac{d}{2} \rceil \choose d-k-1}.
\end{align*}
If $d$ is even, then
\begin{align*}
    \rho(d, d-k-1) &= {\frac{d}{2} \choose d-k-1},\\
    f_{d-1}(C(d,n)) &= \frac{2n}{d}{n - \frac{d}{2} - 1 \choose n-d}.
\end{align*}
If $d$ is odd, then
\begin{align*}
    \rho(d, d-k-1) &= \frac{k+2}{d+1}{\frac{d+1}{2} \choose d-k-1},\\
    f_{d-1}(C(d,n)) &= 2{n - \frac{d-1}{2} - 1 \choose n-d}.
\end{align*}
Thus, in either case,
\[
    f_k(C(d,n)) - f_k(C(d-2,n)) = \rho(d, d-k-1)f_{d-1}(C(d,n)).\qedhere
\]
\end{proof}

Using Lemma \ref{cyclic}, we can prove our desired tightness result.

\begin{theorem}
    \label{tight}
    Let $d$ be a positive integer and $0 \leq k \leq d-1$. There exist infinite families of $d$-polytopes $\{P_n\}, \{Q_n\}$ such that for all $\varepsilon > 0$ and $n$ sufficiently large,
    \begin{align*}
        \frac{f_k(P_n)}{f_0(P_n)} &< \rho(d, k) + \varepsilon,\\
        \frac{f_k(Q_n)}{f_{d-1}(Q_n)} &< \rho(d, d-k-1) + \varepsilon.
    \end{align*}
\end{theorem}

\begin{proof}
The proof is trivial for $d=1$. Fix $d \geq 2$ and $0 \leq k \leq d-1$, and observe that for all $n > d$,
\begin{align*}
    f_{d-1}(C(d,n)) &> \frac{(n-d)^{\lceil (d-1)/2 \rceil}}{\lceil\frac{d-1}{2}\rceil!},\\
    f_k(C(d-2,n)) &< \frac{k+2}{n-k-1} \biggl\lfloor \frac{d-2}{2} \biggr \rfloor n^{\lceil (d-1)/2 \rceil}.
\end{align*}
Thus,
\[
    \lim_{n \to \infty} \frac{f_k(C(d-2,n))}{f_{d-1}(C(d,n))} = 0.
\]
By Lemma \ref{cyclic}, it follows that
\[
    \lim_{n \to \infty} \frac{f_k(C(d,n))}{f_{d-1}(C(d,n))} = \rho(d, d-k-1).
\]
Accordingly, if ${C(d,n)}^\star$ is the polar dual of $C(d,n)$, then
\[
    \lim_{n \to \infty} \frac{f_k({C(d,n)}^*)}{f_{0}({C(d,n)}^*)} = \rho(d, k).\qedhere
\]
\end{proof}


\section{The Search for Stronger Bounds}
\label{s:nonlinear}

While Theorem \ref{bound} gives the tightest linear bounds on $f_k(P)$ in terms of $f_0(P)$ or $f_{d-1}(P)$, there is still room for improvement via nonlinear bounds. In this section, we will use Grassmann angles to prove a stronger version of Theorem \ref{bound} (\S \ref{ss:grassmann_nonlinear}). We will then prove a connected result relating the face numbers of $P$ to a projection in codimension two (\S \ref{ss:combinatorial}).

\subsection{Grassmann Angles and Nonlinear Bounds}
\label{ss:grassmann_nonlinear}

\begin{proposition}
    \label{grassmann}
    For all $d$-polytopes $P$ and $0 \leq k \leq d-1$,
    \[
        f_k(P) \geq \rho(d,d-k-1)f_{d-1}(P) + \gamma_k^2(P).
    \]
\end{proposition}

\begin{proof}
    Let $P$ be a $d$-polytope and $0 \leq k \leq d-1$. Let $\mathcal{F}$ be the set of facets of $P$ and $\mathcal{G}$ the set of $k$-dimensional faces. Summing Theorem \ref{deficiency} over $\mathcal{G}$, we find
    \[
        \gamma^2_k(P) = f_k(P) - \sum_{\substack{F \in \mathcal{F} \\ {G \in \mathcal{G}}}}\varphi(F,G) = f_k(P) - \sum_{F \in \mathcal{F}}\varphi_k(F).
    \]
    By Theorem \ref{angles}, it follows that
    \[
        \gamma^2_k(P) \leq f_k(P) - \rho(d,d-k-1)f_{d-1}(P). \qedhere
    \]
\end{proof}

By Lemma \ref{grass_expectation}, we know $\gamma_k^2(P) \geq f_k(\Delta^{d-2}) = {d-1 \choose k+1}$. Combining this with Proposition \ref{grassmann} yields the following bounds on $f_k(P)$ in terms of $f_0(P)$ and $f_{d-1}(P)$, a slight improvement on Theorem \ref{bound}.

\begin{corollary}
    Let $P$ be a $d$-polytope, and suppose $0 \leq k \leq d-1$. Then
    \begin{align*}
        f_k(P) &\geq \rho(d,k)f_0(P) + {d-1 \choose k-1},\\
        f_k(P) &\geq \rho(d,d-k-1)f_{d-1}(P) + {d-1 \choose k+1}.
    \end{align*}
\end{corollary}

\subsection{Face Number Relations from a Fixed Projection}
\label{ss:combinatorial}

There is a deeper, combinatorial identity hidden behind Proposition \ref{grassmann}. This identity describes the face numbers of a polytope $P$ by its behavior under any codimension two orthogonal projection.

\begin{definition}
    Let $P$ be a $d$-polytope and $S$ an $m$-dimensional subspace of $\mathds{R}^d$. We say $S$ is in \emph{general position with respect to $P$} if for all $0 \leq k \leq d-1$, all $k$-dimensional faces $G$ of $P$, and all points $z$ in the relative interior of $G$,
    \[
        \dim((S+z) \cap G) = \max\{0,m+k-d\},
    \]
    where $S+z$ is the translation of $S$ by the vector $z$.
\end{definition}

\begin{proposition}
\label{combinatorial}
    Let $P$ be a $d$-polytope with $d \geq 3$, $S \subset \mathds{R}^d$ a two-dimensional subspace in general position with respect to $P$, and $\pi:\mathds{R}^d \to S^\perp$ the orthogonal projection map. Let $\mathcal{F}$ be the set of facets of $P$. Then for all $0 \leq k \leq d-3$,
    \[
        f_k(P) - f_k(\pi(P)) = \frac{1}{2}\sum_{F \in \mathcal{F}} \bigl[f_k(F) - f_k(\pi(F)) \bigr].
    \]
\end{proposition}

\begin{proof}
    Let $\mathcal{G}$ be the set of $k$-dimensional faces $G$ of $P$ such that $\pi(G)$ is not a face of $\pi(P)$. Let $\mathcal{E}$ be the set of pairs $(F,G)$ such that $F$ is a facet of $P$, $G$ is a $k$-dimensional face of $F$, and $\pi(G)$ is not a face of $\pi(F)$. Then
    \begin{align*}
        |\mathcal{G}| &= f_k(P) - f_k(\pi(P)),\\
        |\mathcal{E}| &= \sum_{F \in \mathcal{F}}\bigl[f_k(F) - f_k(\pi(F)) \bigr].
    \end{align*}
    
    Suppose $G \in \mathcal{G}$, and let $z$ be an arbitrary point in the relative interior of $G$. Since $S$ is in general position with respect to $P$, each facet $F$ of $P$ containing $G$ has $\dim((S+z) \cap F) \leq 1$. If $(F,G) \in \mathcal{E}$, then $\dim((S+z) \cap F) = 1$; if not, then $(S+z) \cap F = \{z\}$. Meanwhile, for all faces $K$ of $P$ containing $G$ which are not facets, $(S+z) \cap K = \{z\}$.

    Since $\pi(G)$ is not a face of $\pi(P)$, $S+z$ must intersect the relative interior of $P$. Thus, $(S+z) \cap P$ is a polygon. For all facets $F$ of $P$ containing $G$, $\dim((S+z) \cap F) = 1$ if and only if $(S+z) \cap F$ is an edge of $(S+z) \cap P$ containing $z$. Thus, there are exactly two facets $F, F'$ of $P$ with $(F,G),(F',G) \in \mathcal{E}$.

    Now consider an arbitrary $(F,G) \in \mathcal{E}$. We know $\dim \pi(F) = \dim \pi(P) = d-2$, so $\pi(F)$ cannot be a proper face of $\pi(P)$. Thus, $\operatorname{relint}(\pi(G)) \subset \operatorname{relint}(\pi(F)) \subseteq \operatorname{relint}(\pi(P))$. It follows that $G \in \mathcal{G}$.

    We have shown that $(F,G) \mapsto G$ is a well-defined map from $\mathcal{E}$ to $\mathcal{G}$, under which the preimage of each $G \in \mathcal{G}$ has size exactly two. Thus,
    \[
        f_k(P) - f_k(\pi(P)) = \frac{1}{2}\sum_{F \in \mathcal{F}} \bigl[f_k(F) - f_k(\pi(F)) \bigr].\qedhere
    \]
\end{proof}

Let $P$ be a polytope and $F$ a facet of $P$. Let $S, \pi$ be as defined in Proposition \ref{combinatorial}, and let $T$ be the image of the orthogonal projection of $S^\perp$ onto $\operatorname{aff}(F)$, so $T$ has codimension one in $\operatorname{aff}(F)$. Then $\pi(F)$ is combinatorially equivalent to the orthogonal projection of $F$ onto $T$. If $S$ is chosen uniformly at random, then the resulting $T \subset \operatorname{aff}(F)$ obeys a uniform distribution as well. It follows by Theorem \ref{feldman} that
\[
    E\bigl[f_k(F) - f_k(\pi(F)) \bigr] = 2\varphi_k(F).
\]
Thus, applying linearity of expectation to Proposition \ref{combinatorial} yields exactly the equation of Proposition \ref{grassmann}.

\begin{remark}
Proposition \ref{combinatorial} illustrates a remarkable property of the cyclic polytope $C(d,n)$. Let $\pi:\mathds{R}^d \to \mathds{R}^{d-2}$ be the projection onto the first $d-2$ coordinates, so $\pi(C(d,n)) = C(d-2,n)$. Then by Proposition \ref{combinatorial} and Lemma \ref{cyclic},
\[
    \sum_{F \in \mathcal{F}} \bigl[f_k(F) - f_k(\pi(F)) \bigr] = 2\rho(d,d-k-1)f_{d-1}(C(d,n)),
\]
where $\mathcal{F}$ is the set of facets of $C(d,n)$. Thus, all facets $F \in \mathcal{F}$ simultaneously attain
\[
    f_k(F) - f_k(\pi(F)) = 2\rho(d,d-k-1),
\]
the minimum possible value for any $(d-1)$-polytope $F$ and orthogonal projection $\pi$ onto a codimension one subspace in general position. This property helps explain why cyclic polytopes have among the lowest face numbers for polytopes with a fixed number of facets, and optimistically, it may be a clue toward proving the Generalized Upper Bound Conjecture.
\end{remark}

\section{Acknowledgements}
The author would like to thank Isabella Novik for her dedicated mentorship and guidance in writing this paper. The author was partially supported by a graduate fellowship from NSF grants DMS-1953815 and DMS-2246399.


\bibliography{bibliography}
\bibliographystyle{plain}

\end{document}